\documentclass[12pt]{amsart}
\usepackage{amscd}
\usepackage{amsmath}
\usepackage{verbatim}
\usepackage{longtable}
\usepackage{stmaryrd}
\usepackage{mathtools}
\usepackage{yfonts}
\usepackage[hypertex]{hyperref}

\usepackage[all, cmtip,arrow]{xy}

\usepackage{pb-diagram,pb-xy}

\usepackage{amssymb}

\numberwithin{equation}{section}

\newtheorem{thm}[equation]{Theorem}
\newtheorem{prop}[equation]{Proposition}

\newtheorem{lem}[equation]{Lemma}
\newtheorem{cor}[equation]{Corollary}

\theoremstyle{definition}
\newtheorem{rem}[equation]{Remark}
\newtheorem{example}[equation]{Example}

\newtheorem{ntt}[equation]{}

\newcommand{\zz}{\mathbb{Z}}
\newcommand{\ta}{\mathbb{T}}

\newcommand{\qq}{\mathbb{Q}}

\newcommand{\laz}{\mathbb{L}}
\newcommand{\pp}{\mathbb{P}}
\newcommand{\ff}{\mathbb{F}}

\newcommand{\la}{\langle}
\newcommand{\ra}{\rangle}

\DeclareMathOperator{\End}{\mathrm{End}}

\DeclareMathOperator{\codim}{\mathrm{codim}}

\DeclareMathOperator{\OGr}{\mathrm{OGr}}

\newcommand{\pr}{\mathrm{pr}}
\newcommand{\res}{\mathrm{res}}

\newcommand{\pt}{\mathrm{pt}}

\DeclareMathOperator{\Spec}{\mathrm{Spec}}
\DeclareMathOperator{\CH}{\mathrm{CH}}

\newcommand{\Mot}{\mathcal{M}}

\newcommand{\cc}{\mathfrak{b}}

\newcommand{\td}{\mathrm{td}}
\newcommand{\itd}{\mathrm{itd}}
\newcommand{\FGL}{\mathrm{FGL}}
\newcommand{\ch}{\mathfrak{c}}

\title{Morava $K$-theory of twisted flag varieties}
\author{Victor Petrov, Nikita Semenov}
\date{}
\thanks{The authors gratefully acknowledge the support of the Sonderforschungsbereich/Transregio 45 ``Periods,
moduli spaces, and arithmetic of algebraic varieties'' (Bonn-Essen-Mainz).
The first author is supported by the Chebyshev Laboratory (Department of Mathematics and Mechanics,
St. Petersburg State University) under RF Government grant 11.G34.31.0026 and by JSC ``Gazprom Neft'',
and is partially supported by RFBR 13-01-00429.
The second author also acknowledges the support of Universit\'e Paris~13.}

\begin{document}
\maketitle

\begin{abstract}
In the present article we prove some results about the Morava $K$-theory. In particular,
we construct an operation from the Morava $K$-theory to the Chow theory analogous to the
second Chern class for Grothendieck's $K^0$-theory. Furthermore,
we investigate ordinary and equivariant oriented cohomology theories in the sense of Levine-Morel of projective quadrics,
and discuss the Rost motives.
\end{abstract}


\section{Introduction}  

The concept of oriented cohomology theories is well-known in {\it algebraic topology}.
Levine and Morel introduced in \cite{LM} a universal oriented cohomology theory
in {\it algebraic geometry} --- the algebraic cobordism. Using it one can define different other
cohomology theories by the change of coefficients.

Namely, like in topology every oriented cohomology theory is equipped with a formal group law. For example,
for the Chow theory this is the additive formal groups law
and for Grothendieck's $K^0$-theory this is the multiplicative formal group law. Further,
taking the Lubin-Tate formal group law one can define the Morava $K$-theories in algebraic geometry.
We remark that in algebraic topology the Morava $K$-theory is a well-established theory (see \cite{Rav}).

Due to works of Levine, Morel, Panin and Smirnov there exist the Riemann-Roch-type theorems
which can be used to investigate different aspects of oriented cohomology theories.
Moreover, recently an equivariant version of algebraic cobordism has appeared and, hence,
one can define arbitrary equivariant oriented theories in algebraic geometry.

The algebraic cobordism theory has found applications
in the theory of linear algebraic groups and quadratic forms mainly due to the works of Vishik
(e.g. in the construction of fields with the $u$-invariant of the form $2^r+1$, $r\ge 3$).
Moreover, the classical theories, like the Chow theory and $K^0$ are known to give strong
applications in the classification of central simple algebras, quadratic forms and algebraic groups (e.g. the index
reduction formulae of Merkurjev, Panin, Wadsworth; diverse results on quadratic forms due to Karpenko and Vishik --- see
e.g. the introduction in \cite{GPS14}).

In the present article we investigate the Morava $K$-theory of algebraic varieties and provide
a tool to compute the motive with respect to an arbitrary oriented cohomology theory of projective quadrics.
We start in Section~\ref{sec2} summarizing some general results related to oriented cohomology theories.
Section~\ref{sec3} is devoted to a construction of a surjective operation from the Morava $K$-theory
to the Chow theory analogous to the second Chern class for $K^0$. This operation turns out to be additive
when we consider it as an operation from the Morava $K$-theory to the Witt group of the Chow ring associated
with the Lubin-Tate formal group law.

In Section~\ref{sec4} we discuss the Morava-Rost motives and in the remaining sections we introduce an
algorithm to compute the motive of a projective quadric with respect to an arbitrary oriented cohomology theory.
This algorithm essentially involves the equivariant theories.

\section{Definitions and notation}\label{sec2}
We follow \cite{Ful}, \cite{LM}, \cite{Pa03}, \cite{Haz}, \cite{Rav}.

\begin{ntt}[Oriented cohomology theories and motives]
Let $F$ be a field of characteristic $0$. We denote by $A^*$ an oriented cohomology theory in the sense of
Levine-Morel over $F$ which satisfies the localization axiom and is generically constant (\cite[Def.~4.4.1, Def.~4.4.6]{LM},
cf. \cite{Pa03}, \cite{Sm06}).
In particular, we consider $A^*=\CH^*$ the Chow ring modulo rational equivalence,
$A^*=K^0[v_1,v_1^{-1}]$ the Grothendieck group of locally free coherent sheaves,
and $A^*=\Omega^*$ the algebraic cobordism.

By \cite[Thm.~1.2.6]{LM} the algebraic cobordism is a universal oriented cohomology theory and there is a
(unique) morphism of theories $\Omega^*\to A^*$ for any oriented cohomology theory $A^*$.

Each oriented cohomology theory $A$ is equipped with a $1$-dimensional commutative formal group law $\FGL_A$. For $\CH^*$
this is the additive formal group law, for $K^0$ the multiplicative formal group law and for $\Omega$ the universal
formal group law. Moreover, these theories are universal for the respective formal group laws.

For a theory $A^*$ we consider the category of $A^*$-motives with coefficients in a commutative ring $R$,
which is defined in the same way as the category of Grothendieck's Chow motives with $\CH^*$
replaced by $A^*\otimes_\zz R$ (see \cite{Ma68}). In the present
article the ring $R$ is $\zz$, $\zz_{(p)}$, or $\ff_p$ for a prime number $p$.
\end{ntt}

\begin{ntt}[Morava-like theories]
For a prime number $p$ and a natural number $n$ we consider the $n$-th Morava $K$-theory $K(n)$ with respect to $p$.
Note that we do not include $p$ in the notation.
We define this theory as the universal oriented cohomology theory for the Lubin-Tate formal group law of height $n$
with the coefficient ring $\zz_{(p)}[v_n,v_n^{-1}]$ (see Section~\ref{ltfgl}).

For a variety $X$ over $F$ one has $$K(n)(X)=\Omega(X)\otimes_\laz\zz_{(p)}[v_n,v_n^{-1}],$$
and $v_n$ is a $\nu_n$-element in the Lazard ring $\mathbb{L}$ (see e.g. \cite[Def.~2.3]{Sem13}). The degree of $v_n$
is negative and equals $-(p^n-1)$.
In particular, $K(n)(\Spec F)=\zz_{(p)}[v_n,v_n^{-1}]$. We remark that usually one considers the Morava
$K$-theory with the coefficient ring $\ff_p[v_n,v_n^{-1}]$.

If $n=1$ and $p=2$, one has $K(1)(X)=K^0(X)[v_1,v_1^{-1}]\otimes\zz_{(2)}$, since the Lubin-Tate formal
group law is isomorphic to the multiplicative formal group law in this case.

Moreover, sometimes we consider the Brown-Peterson cohomology $BP$ with coefficient ring $\zz_{(p)}[v_1,v_2,\ldots]$.
The $n$-th Morava $K$-theory can be obtained from $BP$ by sending all $v_i$ with $i\ne n$ to zero and localizing
at $v_n$. Besides this, we consider the connective Morava $K$-theory $CK(n)$,
where we do not invert $v_n$. The coefficient ring of $CK(n)$ equals, thus, $\zz_{(p)}[v_n]$.
\end{ntt}

\begin{ntt}[Lubin-Tate formal group law]\label{ltfgl}
We construct the formal group law for the $n$-th Morava $K$-theory modulo $p$ following \cite{Haz}.
The logarithm of the formal group law of the Brown-Peterson cohomology equals
$$l(t)=\sum_{i\ge 0}m_it^{p^i},$$
where $m_0=1$ and the remaining variables $m_i$ are related to $v_j$ as follows:
$$m_j=\frac 1p\cdot\big(v_j+\sum_{i=1}^{j-1}m_iv_{j-i}^{p^i}\big).$$
Let $e(t)$ be the compositional inverse of $l(t)$.
The Brown-Peterson formal group law is given by $e(l(x)+l(y))$.

The $n$-th Morava formal group law is obtained from the $BP$ formal group law by sending all $v_j$ with $j\ne n$
to zero. Modulo the ideal $I$ generated by $p, x^{p^n}, y^{p^n}$ the formal group law for the $n$-th Morava
$K$-theory equals
$$\FGL_{K(n)}(x,y)=x+y-v_n\sum_{i=1}^{p-1}\frac 1p\binom pi x^{ip^{n-1}}y^{(p-i)p^{n-1}}\mod I.$$
In the same way, the logarithm of $K(n)$ is obtained from $l(t)$ by sending all $v_j$ with $j\ne n$
to zero.
\end{ntt}

\begin{ntt}[Generalized Witt vectors]
Let $S$ be a commutative ring of characteristics $0$ and
$$l(x)=\sum_{i=1}^\infty a_ix^i\in S\otimes\qq[[x]],\qquad a_1=1$$
the logarithm of a formal group law over $S$. Following \cite[15.3]{Haz} we define an abelian group of
Witt vectors associated with $l(x)$ as follows.

First we define polynomials $w_n(z)$ in $z_1,\ldots,z_n$ as $$w_n(z)=\sum_{d\mid n}a_{n/d}z_d^{n/d}.$$
Next define the polynomials $\Sigma_n(x_1,\ldots,x_n;y_1,\ldots,y_n)$ as
$$w_n(\Sigma_1(x;y),\ldots,\Sigma_n(x;y))=w_n(x)+w_n(y).$$
Then the coefficients of $\Sigma_i$ lie in $S$ (not just in $S\otimes\qq$).

Consider the set $W(S)$ of vectors of infinite length with coordinates in $S$ and define the structure of an
abelian group on $W(S)$ by $$(a_1,a_2,\ldots)+_{W(S)}(b_1,b_2,\ldots)=(\Sigma_1(a;b),\Sigma_2(a;b),\ldots)$$

The group $W(S)$ is the group of generalized Witt vectors. By construction it depends on $l(x)$. Moreover, note that
the addition in $W(S)$ involves the multiplication in the ring $S$. The classical Witt vectors is a particular
case of this general construction, when one starts with the multiplicative formal group law.

\begin{example}\label{witte}
Consider the logarithm $l(x)$ of the Morava $K$-theory mod $p$ and send additionally $v_n$ to $1$.
A straightforward computation shows that $$l(x)=x+\tfrac 1p\cdot x^{p^n}+\mathcal{O}(x^{2p^n-1}),$$
where $\mathcal{O}$ is the Landau symbol.
Then
$$w_1(z_1)=z_1, w_2(z_1,z_2)=z_2,\ldots, w_{p^n-1}(z_1,\ldots,z_{p^n-1})=z_{p^n-1},$$
$$w_{p^n}(z_1,\ldots,z_{p^n})=z_{p^n}+\tfrac 1p\cdot z_1^{p^n},$$
$$\Sigma_1(x;y)=x_1+y_1,\Sigma_2(x;y)=x_2+y_2,\ldots,\Sigma_{p^n-1}(x;y)=x_{p^n-1}+y_{p^n-1},$$
$$\Sigma_{p^n}(x;y)=x_{p^n}+y_{p^n}+\tfrac 1p(x_1^{p^n}+y_1^{p^n}-(x_1+y_1)^{p^n}).$$
\end{example}
\end{ntt}

\begin{ntt}[Riemann-Roch theorem]\label{rro}
Let $A$ be an oriented cohomology theory, which is universal for its formal group law $\FGL_A$.

Consider a morphism $$\varphi\colon\FGL_{\CH_{\qq}}\to\FGL_A$$
of formal group laws, i.e., $\FGL_A(\varphi(x),\varphi(y))=\varphi(x+y)$.
We always choose the power series $\varphi(t)$ in such a way that its coefficient at $t$ equals $1$.
This morphism induces a morphism of oriented cohomology theories
$$\ch_\varphi\colon A^*\to\CH^*\otimes A^*(F)_\qq.$$

Let $E$ be a vector bundle over a smooth variety $X$ and let $\alpha_i$ be its roots.
Then the Todd class is defined as $\td_\varphi(E)=\prod\td_\varphi(\alpha_i)$ and
$\td_\varphi(t)=\dfrac{t}{\varphi(t)}$. The inverse Todd class
of $E$ is defined as $\itd_\varphi(E)=\prod\itd_\varphi(\alpha_i)$ and $\itd_\varphi(t)=\dfrac{\varphi(t)}{t}$.

For the $n$-th Morava $K$-theory the function
$$\varphi\colon\FGL_{\CH}\to\FGL_{K(n)}$$ is obtained from $e(t)$ of Section~\ref{ltfgl} by sending all $v_j$ with $j\ne n$
to zero.

The following proposition and its corollary is the Riemann-Roch theorem for general oriented cohomology
theory developed by Levine, Morel, Panin and Smirnov in \cite{Sm06}, \cite[Thm.~2.5.3, 2.5.4]{Pa03}, \cite[Ch.~4]{LM}.

\begin{prop}
Let $A$ be an oriented cohomology theory, which is universal for its formal group law.
\begin{enumerate}
\item Let $i\colon Y\to X$ be a closed embedding of smooth varieties with the normal bundle $N$ over $Y$ and $\alpha\in A(Y)$.
Then $$i_*(\ch_\varphi(\alpha)\cdot\itd_\varphi(N))=\ch_\varphi(i_*(\alpha)).$$
\item Let $f\colon Y\to X$ be a morphism of smooth projective varieties and $\alpha\in A(Y)$. Then
$$f_*(\ch_\varphi(\alpha)\cdot\td_\varphi(T_Y))=\ch_\varphi(f_*(\alpha))\cdot\td_\varphi(T_X),$$
where $T_X$ and $T_Y$ denote the tangent bundle over $X$ and $Y$.
\end{enumerate}
\end{prop}                                                        

\begin{cor}\label{cor1}
Let $X$ be a smooth projective variety, $\pi\colon X\to\Spec F$ the
structural morphism, and $\alpha\in A^*(X)$. Then
$$\deg(\ch_\varphi(\alpha)\cdot\td_\varphi(T_X))=\ch_\varphi(\pi_*(\alpha)).$$

For the Chern class $c_1^A(L)\in A^1(X)$ of a line bundle $L$ over $X$ we have
$$\ch_\varphi(c_1^A(L))=\varphi(c_1^{\CH}(L)).$$
\end{cor}
\begin{proof}
We illustrate the method proving the last identity with the first Chern class of $L$.
Let $s\colon X\to L$ be the zero section of the line bundle $L$. Then $c_1^A(L)=s^*_As_*^A(1_X)$,
where $s^*_A$ and $s_*^A$ denote the pullback and the pushforward for the theory $A$.

Then $\ch_\varphi(c_1^A(L))=\ch_\varphi(s^*_As_*^A(1_X))=s^*_{\CH}\ch_\varphi(s^A_*(1_X))=
s^*_{\CH}s_*^{\CH}(\itd_\varphi(N))$,
where $N$ is the normal line bundle to the zero section. By \cite[Ex.~2.5.5]{Ful} this bundle is the restriction
of $\mathcal{O}_L(X)$ to $X$ and $s_*^{\CH}(1)=[X]=c_1^{\CH}(\mathcal{O}_L(X))$.

Write $\varphi(t)=\sum_{i\ge 1}a_it^i$ as a power series. Then
\begin{align*}
\itd_\varphi(N)=\itd_\varphi(c_1^{\CH}(N))=
\sum_{i\ge 1}a_i(c_1^{\CH}(N))^{i-1}=\sum_{i\ge 1}a_i(c_1^{\CH}(s^*\mathcal{O}_L(X)))^{i-1}\\
=\sum_{i\ge 1}a_i(s^*_{\CH}(c_1^{\CH}(\mathcal{O}_L(X))))^{i-1}=
\sum_{i\ge 1}a_is^*_{\CH}((c_1^{\CH}(\mathcal{O}_L(X)))^{i-1}).
\end{align*}

Therefore
\begin{align*}
s^*_{\CH}s_*^{\CH}(\itd_\varphi(N))=s^*_{\CH}s_*^{\CH}(\sum_{i\ge 1}a_is^*_{\CH}((c_1^{\CH}(\mathcal{O}_L(X)))^{i-1}))\\
=s^*_{\CH}(\sum_{i\ge 1}a_is_*^{\CH}(s^*_{\CH}((c_1^{\CH}(\mathcal{O}_L(X)))^{i-1})))=
s^*_{\CH}(\sum_{i\ge 1}a_i(c_1^{\CH}(\mathcal{O}_L(X)))^{i-1}s_*^{\CH}(1))\\
=s^*_{\CH}(\sum_{i\ge 1}a_i[X]^{i-1}s_*^{\CH}(1))=s^*_{\CH}(\sum_{i\ge 1}a_i(s_*^{\CH}(1))^i)=
\sum_{i\ge 1}a_is^*_{\CH}((s_*^{\CH}(1))^i)\\
=\sum_{i\ge 1}a_i(s^*_{\CH}(s_*^{\CH}(1)))^i=\sum_{i\ge 1}a_i(c_1^{\CH}(L))^i=\varphi(c_1^{\CH}(L)),
\end{align*}
where we used the projection formula and the multiplicativity of pullbacks.
\end{proof}
\end{ntt}

\begin{ntt}[Euler characteristic]\label{eulerc}
The Euler characteristic of a smooth projective variety $X$ with respect to an oriented cohomology theory $A^*$ is defined
as $\pi_*^A(1_X)\in A^*(\Spec F)$, where $\pi\colon X\to\Spec F$ is the structural map.

The Euler characteristic can be computed using the Riemann-Roch theorem. E.g., for $A=K^0[v_1,v_1^{-1}]$
the Euler characteristic of $X$ equals $$v_1^{\dim X}\cdot\sum(-1)^i\dim H^i(X,\mathcal{O}_X)$$
see \cite[Ch.~15]{Ful}. If $X$ is geometrically cellular, then this element equals $v_1^{\dim X}$.

For the Morava $K$-theory $K(n)$ and a variety $X$ of dimension $d=p^n-1$
the Euler characteristic modulo $p$ equals the element $v_n\cdot u\cdot s_d$ for some $u\in\ff_p^\times$,
where $s_d$ is the {\it Milnor number} of $X$ (see \cite[Sec.~4.4.4]{LM}, \cite[Sec.~2.2]{Sem13}). If $\dim X$ is not divisible by $p^n-1$, then the Euler characteristic of $X$ equals zero modulo $p$
(see \cite[Prop.~4.4.22]{LM}).
\end{ntt}

\begin{ntt}[Rost nilpotence for oriented cohomology theories]\label{rostnil}
Let $A$ be an oriented cohomology theory and consider the category of $A$-motives over $F$.
For a smooth projective variety $X$ over $F$ we denote by $\Mot(X)$ its motive ($A$-motive).
We say that the Rost nilpotence principle holds for $X$,
if the kernel of the restriction homomorphism $$\End(\Mot(X))\to\End(\Mot(X_E))$$ consists of nilpotent correspondences for all
field extensions $E/F$.

Usually Rost nilpotency is formulated for Chow motives. By \cite[Sec.~8]{CGM05} it holds for
all twisted flag varieties. Note that the proof of \cite{CGM05} works for $A$-motives of twisted flag varieties
for all oriented cohomology theories $A$ satisfying the localization property.
\end{ntt}

\section{Some operations in the Morava $K$-theory}\label{sec3}

A straightforward computation shows the following lemma:
\begin{lem}\label{lem1}
Let $p$ be a prime number and $n$ be an integer.
Consider the power series $\varphi\colon\FGL_{\CH}\to\FGL_{K(n)}$ of Section~\ref{rro}. Then
$$\varphi(t)=t-\tfrac 1p\cdot v_nt^{p^n}+\mathcal{O}(t^{2p^n-1}),$$
where $\mathcal{O}$ is the Landau symbol.
\end{lem}

Let now $\ch:=\ch_{\varphi}=\ch^{K(n)}\colon K(n)^*\to\CH^*\otimes\qq[v_n,v_n^{-1}]$ be the respective morphism
of theories and $\ch_i$ its codimension $i$ component. We remark that if $n=1$ and $p=2$, then $\ch$ is
the classical Chern character from $K^0$ to $\CH_{\qq}$.

Following notation of Haution \cite{Ha12} define $$\CH_{\zz_{(p)}\subset\qq}[v_n,v_n^{-1}]=
\mathrm{Im}(\CH\otimes\zz_{(p)}[v_n,v_n^{-1}]\to\CH\otimes\qq[v_n,v_n^{-1}]).$$

\begin{thm}
\begin{enumerate}
\item For $i\le p^n-1$ the image of $\ch_i$ lies in $\CH_{\zz_{(p)}\subset\qq}[v_n,v_n^{-1}]$.
\item The image of $\cc:=\ch_{p^n}+\tfrac 1p\cdot\ch_1^{p^n}$ also lies in $\CH_{\zz_{(p)}\subset\qq}[v_n,v_n^{-1}]$.
\item For $i\le p^n-1$ the maps $\ch_i$ are onto $\CH^i_{\zz_{(p)}\subset\qq}[v_n,v_n^{-1}]$
and the map $\cc$ is onto $\CH^{p^n}_{\zz_{(p)}\subset\qq}[v_n,v_n^{-1}]$.
\item The map $(\ch_1,\ldots,\ch_{p^n-1},-\cc,0,\ldots)$ is a group homomorphism
$$K(n)\to W(\CH^{*}_{\zz_{(p)}\subset\qq}[v_n,v_n^{-1}]),$$ where $W$ are the generalized Witt vectors associated with the Lubin-Tate formal group law
of Example~\ref{witte}.
\end{enumerate}
\end{thm}

\begin{proof}
Let $X$ be a smooth projective variety over $F$.

For each closed subvariety $Z\subset X$ let $\widetilde Z\to Z$ be its resolution
of singularities.
By \cite[Thm.~4.4.7]{LM} $K(n)(X)$ is generated as a $K(n)(\pt)$-module by the classes
$[\widetilde Z\to X]_{K(n)}$.
The morphisms $\ch_i$ are additive and
\begin{align*}
\cc(\alpha+\beta)=\cc(\alpha)+\cc(\beta)+\tfrac 1p((\ch_1(\alpha)+\ch_1(\beta))^{p^n}-\ch_1(\alpha)^{p^n}-\ch_1(\alpha)^{p^n}).
\end{align*}

Moreover, $\ch$ is $\zz_{(p)}[v_n,v_n^{-1}]$-linear.
Therefore it is sufficient to prove the integrality of $\ch_i$ and $\cc$ on generators.
We compute $\ch_i([\widetilde Z\to X]_{K(n)})$ next.

We have by Riemann-Roch
\begin{multline}\label{f1}
\ch_i([\widetilde Z\xrightarrow{f}X]_{K(n)})=
\ch_i(f_*^{K(n)}(1_{\widetilde Z}))\\
=\text{codimension }i\text{ component of }
\td_\varphi(-T_X)\cdot f_*(\td_\varphi(T_{\widetilde Z})).
\end{multline}
Computing $\td_\varphi$ using Lemma~\ref{lem1}, this implies for $i\le p^n-1$ that
$$\ch_i([\widetilde Z\to X])=
\begin{cases}
[Z],&\text{if } \codim Z=i;\\
0,& \text{otherwise.}
\end{cases}
$$
This proves the first three statements of the theorem for $i\le p^n-1$.

Consider now $\ch_{p^n}$. Formula~\eqref{f1} shows that
$\ch_{p^n}([\widetilde Z\to X])=0$ if $\codim Z\ne 1,p^n$ and
$\ch_{p^n}([\widetilde Z\to X])=[Z]$, if $\codim Z=p^n$. This proves the surjectivity
of $\cc$.

To check the integrality of $\cc$ it sufficies to compute it on divisors.

Let $Z$ be a closed subvariety of $X$ of codimension $1$, $L$ the
respective line bundle over $X$, and $\widetilde Z\to Z$ a resolution of
singularities of $Z$. Denote by $\pr\colon\Omega^*(X)\to\CH^*(X)$ the projection to the Chow theory.
Then $$\pr([\widetilde Z\to X])-\pr(c^\Omega_1(L))=0$$ and, thus, the element
$\alpha:=[\widetilde Z\to X]-c^\Omega_1(L)$ lies in the kernel of $\pr$, which is
equal by \cite[Rem.~4.5.6]{LM} to $\laz_{\ge 1}\Omega(X)$. Since the degree of $\alpha$ equals $1$,
$\alpha$ can be expressed as a linear combination with coefficients in $\laz$
of elements $[\widetilde Y\to X]$, where $\widetilde Y\to Y$ is a resolution
of singularities of a subvariety $Y$ of $X$ of codimenion bigger than $1$.

Projecting $\alpha$ to $K(n)$ reduces the problem of integrality of $\cc$ to its
integrality on classes of codimension bigger than $1$ and to the first Chern classes of line bundles.
But for classes of codimension bigger than $1$ the integrality property of $\cc$ was shown in the first part of the proof.

Therefore it remains to prove our theorem for the first Chern classes of line
bundles. But in this case it follows from Corollary~\ref{cor1}, since for a line bundle $L$ over $X$
$$\ch(c_1^{K(n)}(L))=c_1^{\CH}(L)-\tfrac 1p c_1^{\CH}(L)^{p^n}+\text{higher degree terms}.$$ In particular,
$\ch_{p^n}(c_1^{K(n)}(L))+\tfrac 1p\ch_1^{K(n)}(c_1^{K(n)}(L))^{p^n}=-\tfrac 1p c_1^{\CH}(L)^{p^n}+\tfrac 1p(c_1^{\CH}(L))^{p^n}=0$ is integral.

Finally, it remains to show the additivity of the operation $\mathfrak{C}:=(\ch_1,\ldots,\ch_{p^n-1},-\cc,0,\ldots)$.
But this follows from the explicit formulae of Example~\ref{witte}:
\begin{align*}
&\mathfrak{C}(\alpha)+\mathfrak{C}(\beta)=(\ch_1(\alpha),\ch_2(\alpha),\ldots,-\cc(\alpha),0,\ldots)+_W
(\ch_1(\beta),\ch_2(\beta),\ldots,-\cc(\beta),0,\ldots)\\
&=(\ch_1(\alpha)+\ch_1(\beta),\ch_2(\alpha)+\ch_2(\beta),\ldots,\\
&-\ch_{p^n}(\alpha)-\tfrac 1p\ch_1(\alpha)^{p^n}-\ch_{p^n}(\beta)-\tfrac 1p\ch_1(\beta)^{p^n}+
\tfrac 1p(\ch_1(\alpha)^{p^n}+\ch_1(\beta)^{p^n}-(\ch_1(\alpha)+\ch_1(\beta))^{p^n}),0,\ldots)\\
&=\mathfrak{C}(\alpha+\beta) \text{ for }\alpha,\beta\in K(n)(X).
\end{align*}
\end{proof}        

\begin{cor}
Let $X$ be a smooth projective variety over $F$. Assume that $X$ is geometrically cellular, i.e.,
there exists a field extension $E/F$ such that $X_E$ is a cellular variety,
and assume that the restriction map
$$\mathrm{res}_{K(n)}\colon K(n)(X)\to K(n)(X_E)$$ is an epimorphism. Then the restriction map
$$\mathrm{res}_{\CH^i\otimes\zz_{(p)}}\colon\CH^i(X)\otimes\zz_{(p)}\to\CH^i(X_E)\otimes\zz_{(p)}$$ is an epimorphism
for all $i\le p^n$. 
\end{cor}
\begin{proof}
Since $X_E$ is cellular, $\CH_{\zz_{(p)}\subset\qq}(X_E)=\CH(X_E)\otimes\zz_{(p)}$. Now the surjectivity
follows from the following commutative diagram:
\[
 \xymatrix{
 & K(n)(X)\ar@{->>}[r]^{\res} \ar@{->>}[d]  &  K(n)(X_E)\ar@{->>}[d]    \\
 \CH^i(X)\otimes\zz_{(p)}\ar@{->>}[r]\ar@/_1.5pc/[rr]_{\res} &\CH^i_{\zz_{(p)}\subset\qq}(X)\ar@{->}[r] & \CH^i(X_E)\otimes\zz_{(p)}
 }
 \]
where the vertical arrows are the surjective operations $\ch_i$, $i\le p^{n}-1$, or $\cc$,
which by the construction commute with restriction maps.
\end{proof}

\begin{rem}
The operation $\cc$ can be considered as a higher analog of Chern classes. Indeed, for $K(1)$ with $p=2$
it coincides with $c_1^2-c_2$, where $c_i$ are the usual Chern classes (from $K^0$ to $\CH$).

In \cite[Thm.~6.2]{Vi12} Vishik classifies all additive operations $A\to B$ between two oriented cohomology theries, where
$A$ is universal for the respective formal group law. Nevertheless, the Witt vectors do not form
an oriented cohomology theory (cf. \cite[\S2, Thm.~1(2)]{PSh06}) and, thus, operations $\cc$ do not fit directly into
Vishik's context.
\end{rem}

\section{Rost motives}\label{sec4}

Starting from this section we will develop some methods to compute the Morava $K$-theory.

Let $R_m$ denote the (generalized) Rost motive of a non-zero pure symbol $\alpha\in H^m(F,\mu_p^{\otimes m})$ in the category of Chow motives
with $\zz_{(p)}$-coefficients. By definition $R_m$ is indecomposable and for all field extensions $K/F$ the following
conditions are equivalent:
\begin{enumerate}
\item $(R_m)_K$ is decomposable;
\item $(R_m)_K\simeq\bigoplus_{i=0}^{p-1}\zz_{(p)}(b\cdot i)$ with $b=\frac{p^{m-1}-1}{p-1}$;
\item $\alpha_K=0\in H^m(K,\mu_p^{\otimes m})$.
\end{enumerate}
The fields $K$ from this definition are called splitting fields of $R_m$.

The Rost motives were constructed by Rost and Voevodsky (see \cite{Ro06}, \cite{Vo11}). Namely, for all pure symbols $\alpha$
there exists a smooth projective $\nu_{m-1}$-variety $X$ (depending on $\alpha$) over $F$ such that the Chow motive
of $X$ has a direct summand isomorphic to $R_m$. The variety $X$ is called a {\it norm variety} of $\alpha$.

E.g., if $p=2$ and $\alpha=(a_1)\cup\ldots\cup(a_m)$ with $a_i\in F^\times$, then one can take for $X$
the projective quadric given by the equation $\langle\!\langle a_1,\ldots,a_{m-1}\rangle\!\rangle\perp\langle -a_m\rangle=0$,
where $\langle\!\langle a_1,\ldots,a_{m-1}\rangle\!\rangle$ denotes the Pfister form.

By \cite[Sec.~2]{ViYa07} there is a unique lift of the
Rost motive $R_m$ to the category of $\Omega$-motives and, since $\Omega$ is the universal oriented cohomology theory,
there is a well-defined Rost motive in the category of $A^*$-motives for any oriented cohomology theory $A^*$.
We will denote this $A$-motive by the same letter $R_m$. By $\ta(l)$, $l\ge 0$, we denote the Tate motives in the
category of $A$-motives. If $A=\CH\otimes\zz_{(p)}$, we keep the usual notation $\ta(l)=\zz_{(p)}(l)$.

\begin{prop}
Let $p$ be a prime number, $n$ and $m$ be natural numbers and $b=\frac{p^{m-1}-1}{p-1}$.
For a non-zero pure symbol $\alpha\in H^m(F,\mu_p^{\otimes m})$ consider the respective Rost motive $R_m$. Then
\begin{enumerate}
\item If $n<m-1$, then the $K(n)$-motive $R_m$ is a sum of $p$ Tate motives
$\oplus_{i=0}^{p-1}\ta(b\cdot i)$.
\item If $n=m-1$, then the $K(n)$-motive $R_m$ is a sum of the Tate motive $\ta$
and an indecomposable motive $L$ such that
$$K(n)(L)\simeq(\zz^{\oplus(p-1)}\oplus(\zz/p)^{\oplus{(m-2)(p-1)}})\otimes\zz_{(p)}[v_n,v_n^{-1}].$$
For a field extension $K/F$ the motive $L_K$ is isomorphic to a direct sum of twisted Tate motives
iff it is decomposable and iff the symbol $\alpha_K=0$.
\item If $n>m-1$, then the $K(n)$-motive $R_m$ is indecomposable and its realization equals the group
$\CH(R_m)\otimes\zz_{(p)}[v_n,v_n^{-1}]$.
For a field extension $K/F$ the motive $(R_m)_K$ is decomposable iff $\alpha_K=0$.
In this case $(R_m)_K$ is a sum of $p$ Tate motives.
\end{enumerate}
\end{prop}
\begin{proof}
Denote by $\overline R_m$ the scalar extension of $R_m$ to its splitting field.
By \cite[Prop.~11.11]{Ya12} (cf. \cite[Thm.~3.5, Prop.~4.4]{ViYa07}) the restriction map for the $BP$-theory
\begin{equation}\label{eq1}
\res\colon BP(R_m)\to BP(\overline R_m)=BP(\Spec F)^{\oplus p}
\end{equation}
is injective, and the image equals 
\begin{equation}\label{eq2}
BP(R_m)\simeq BP(\Spec F)\oplus I(p,m-2)^{\oplus (p-1)},
\end{equation}
where $I(p,m-2)$ is the ideal in the ring $BP(\Spec F)=\zz_{(p)}[v_1,v_2,\ldots]$ generated by the elements
$\{p,v_1,\ldots, v_{m-2}\}$.

(1) Assume first that $n<m-1$.
Since the ideal $I(p,m-2)$ contains $v_n$ for $n<m-1$ and $v_n$ is invertible in $K(n)(\Spec F)$, we immediately get that all
elements in $K(n)(\overline R_m)$ are rational, i.e., are defined over the base field. By the properties of the
Rost motives
\begin{equation}\label{formom}
\Omega^l(R_m\times R_m)=\bigoplus_{i+j=l}\Omega^i(R_m)\otimes\Omega^j(\overline R_m)
\end{equation}
for all $l$. Since $\Omega$ is a universal theory, the same formula holds for $BP$ and for $K(n)$.
Therefore all elements in $K(n)(\overline R_m\times \overline R_m)$ are rational, and this gives the
first statement of the proposition.

(2) Assume now that $n=m-1$.
Let $X$ be a norm variety for the symbol $\alpha$. In particular, $\dim X=p^{m-1}-1=p^n-1$.
Since the Morava-Euler characteristic of $X$ equals $u\cdot v_n$ for some
$u\in\zz_{(p)}^\times$ (see Section~\ref{eulerc}), the element $v_n^{-1}\cdot u^{-1}(1\times 1)\in K(n)(X\times X)$
is a projector defining the Tate motive $\ta$. Thus, we get the decomposition $R_m\simeq\ta\oplus L$ for some motive $L$.
We claim that $L$ is indecomposable.

Indeed, by \cite[Thm.~4.4.7]{LM} the elements of
$K(n)^{p^n-1}(R_m\times R_m)$ are linear combinations of elements of the form $v_n^s\cdot [Y\to X\times X]$,
where $Y$ is a resolution of singularities of a subvariety of $X\times X$, and $-s(p^n-1)+\codim Y=p^n-1$.
In particular, $s=0,1,-1$ and $\codim Y=0,p^n-1,2(p^n-1)$.
By formula~\eqref{formom} and by the injectivity of the restriction map for $BP$,
it follows that there are at most three rational projectors in $K(n)(\overline R_m\times \overline R_m)$.
These are the diagonal, the projector $v_n^{-1}\cdot u^{-1}(1\times 1)$ constructed above and their difference
(which defines the motive $L$).
Therefore by Rost nilpotency (see Section~\ref{rostnil}) the motive $L$ is indecomposable over $F$.

Taking the tensor product $-\otimes_{BP(\Spec F)} K(n)(\Spec F)$ with formula~\eqref{eq1} and
using \eqref{eq2} one immediatelly gets the formula for $K(n)(L)$.

(3) The same arguments show that $R_m$ is indecomposable for the Morava $K$-theory
$K(n)$ for $n>m-1$, cf. \cite[Ex.~6.14]{Nes13}.
\end{proof}

\begin{rem}
This proposition demonstrates a difference between $K^0$ and the Morava $K(n)$-theory, when $n>1$. By \cite{Pa94} $K^0$
of all twisted flag varieties is $\zz$-torsion-free. This is not the case for $K(n)$, $n>1$.

Moreover, the same arguments as in the proof of the proposition
show that the connective $K$-theory $CK(1)$ (see \cite{Cai08}) of Rost motives $R_m$ for $m>2$ contains non-trivial $\zz$-torsion.
\end{rem}

\begin{rem}
The Chow groups of the Rost motives are known; see \cite[Thm.~8.1]{KM02}, \cite[Thm.~RM.10]{KM13}, \cite[Cor.~10.8]{Ya12},
\cite[Section~4.1]{Vi07}.
\end{rem}

\section{Geometric filtration on a product of quadrics}\label{geomfilt}
Starting from this section we will introduce a method to decompose the $A$-motive of a projective quadric
for an arbitrary oriented cohomology theory $A$. We remark that it is not sufficient to assume that $A=\Omega$
is universal. E.g., the $\Omega$-motive of a generic quadric is indecomposable (since due to Vishik
the Chow motive of a generic quadric is indecomposable (see \cite[Thms.~3.1, 4.1]{Ka12}) and by \cite[Sec.~2]{ViYa07} this implies that
the $\Omega$-motive of a generic quadric is indecomposable). On the other hand, the $K^0$-motives
of all quadrics are always decomposable (see \cite{Pa94}).

Let $q$ be a regular quadratic form over a field $F$ and
$Q=\{\la u\ra\mid q(u)=0\}$ be the respective
$n$-dimensional smooth projective quadric.
Define $$X=\Big\{\big(\la u\ra,\la v\ra\big)\in Q\times Q\mid q(u,v)=0\Big\}.$$

There is the following filtration of $Q\times Q$.
$$Q\times Q\supset X\supset Q,$$
where $Q\subset X$ is the diagonal embedding.

The projection to the first component
$$p_1\colon (Q\times Q)\setminus X\to Q$$ is an $\mathbb{A}^n$-fibration.
Let $\OGr(1,2,Q)$ denote the Grassmannian of isotropic flags of subspaces
of dimensions $1$ and $2$. The map
\begin{align*}
X\setminus Q&\to\OGr(1,2,Q)\\
\big(\la u\ra,\la v\ra\big)&\mapsto \big(\la u\ra\le\la u, v\ra\big)
\end{align*}
is an $\mathbb{A}^1$-fibration.

Let $\tau_2$ be the tautological vector bundle over $\OGr(1,2,Q)$
of rank $2$. Consider
$$\OGr(1,2,Q)\xleftarrow{\pi}\pp(\tau_2)\xrightarrow{f} Q\times Q$$
$$\big(\la u\ra\le\la u, v\ra\big)\mapsfrom\big(\la u\ra\le\la u, v\ra,\,\la w\ra
\le\la u, v\ra\big)
\mapsto \big(\la u\ra,\la w\ra\big)$$
Note $f(\pp(\tau_2))=X$ and so $f$ is a resolution of singularities of $X$.

By \cite[Proof of Thm.~4.4]{NeZ06} we have
\begin{equation}\label{neza}
\begin{split}
&A^*(Q\times Q)\xleftarrow\simeq A^*(Q)\oplus A^{*-1}(\OGr(1,2,Q))\oplus A^{*-\dim Q}(Q)\\
&p_1^*(x)+f_*\circ\pi^*(y)+i_*(z)\mapsfrom x\oplus y\oplus z
\end{split}
\end{equation}
where $i\colon Q\to Q\times Q$ is the diagonal embedding.

\section{Equivariant theories and morphisms}\label{seceq}
In this section we give general information about {\it equivariant} oriented cohomology theories
following Brion and Krishna (see \cite{Bri97}, \cite{Kr10} and references there).

Let $T$ be a split torus (i.e. a product of several copies of the multiplicative group $\mathbb{G}_m$)
and $X$ a smooth projective $T$-variety, where $T$
acts on $X$ with a finite number of fixed points.

Let $A_T^*$ be an oriented $T$-equivariant cohomology theory which is obtained from the equivariant
cobordism theory $\Omega_T^*$ by a change of the coefficients.
Then there is an injective ring homomorphism
\begin{align}\label{ine}
A_T^*(X)\xhookrightarrow{}\bigoplus_{x\in X^T}A_T^*(x)
\end{align}
induced by the embedding of these fixed points $X^T\subset X$, see \cite[Cor.~7.2]{Kr10}.

\begin{ntt}[pull-back]
Let $X$ and $Y$ be smooth projective $T$-varieties with finite number
of fixed points and let $g\colon X\to Y$ be a $T$-equivariant morphism
of relative dimension $d$.

Then we have a commutative diagram
$$
\xymatrix{
A^*_T(Y)\ar[r]^-{g^*}\ar@{^{(}->}[d]&A^*_T(X)\ar@{^{(}->}[d]\\
\bigoplus_{y\in Y^T}A_T^*(y)\ar@{-->}[r]^-{\tilde g^*}&\bigoplus_{x\in X^T}A_T^*(x)
}
$$
where $\tilde g^*(a)_x=a_{g(x)}$, $x\in X^T$, $a\in \bigoplus_{y\in Y^T}A_T^*(y)$.
\end{ntt}

\begin{ntt}[push-forward]
For every $x\in X^T$ the torus $T$ acts on the
vector space $T_{X,x}$, where $T_{X,x}$ denotes the tangent bundle of $X$ at the point $x$. We assume that this representation of $T$ has no zero
weights and define $c^T_{top}(T_{X,x})$ as the product of all its weights
with multiplicities. Furthermore, we assume that $c^T_{top}(T_{Y,y})\ne 0$
for all $y\in Y^T$.

Then there is a commutative diagram
$$
\xymatrix{
A^*_T(X)\ar[r]^-{g_*}\ar@{^{(}->}[d]&A^{*-d}_T(Y)\ar@{^{(}->}[d]\\
\bigoplus_{x\in X^T}A_T^*(x)\ar@{-->}[r]^-{\tilde g_*}&\bigoplus_{y\in Y^T}A_T^{*-d}(y)
}
$$
where
$$\tilde g_*(a)_y=\sum_{f(x)=y}a_x\cdot\frac{c^T_{top}(T_{Y,y})}{c^T_{top}(T_{X,x})},$$
$y\in Y^T$, $a\in\bigoplus_{x\in X^T}A_T^*(x)$.

Let now $y\in Y^T$ and consider the Cartesian square
$$\xymatrix{
Z\ar[r]^-{j}\ar[d]^-{\zeta}&X\ar[d]^-{g}\\
y\ar@{^{(}->}[r]&Y
}$$
Assume that the fiber $Z:=g^{-1}(y)$ is reduced and smooth. Then
\begin{align}\label{fff}
\tilde g_*(a)_y=\zeta_*\big(j^*(a)\cdot c^T_{top}\big(\frac{T_{Y,y}}{d_xg(T_{X,x})}\big)\big)
\end{align}
for all $a\in\bigoplus_{x\in X^T}A_T^*(x)$.
\end{ntt}

\section{Equivariant computations}
Now we apply the previous considerations to the geometric filtration on $Q\times Q$ constructed in Section~\ref{geomfilt}.

\begin{ntt}[Description of the method]\label{method}
Let $Q$ be an arbitrary smooth projective quadric over $F$. Our goal is to find
a decomposition of the $A$-motive of $Q$, i.e., to construct projectors in $A(Q\times Q)$.

By Rost nilpotency it sufficies to construct rational projectors in $A(\overline Q\times\overline Q)$,
where $\overline Q$ denotes the extension of scalars to a splitting field of the quadric $Q$.

To construct rational projectors we use formula~\eqref{neza} which is compatible with scalar extensions.
Namely, we start with rational elements in $A(\overline Q)$ and $A(\OGr(1,2,\overline Q))$ (e.g. with Chern classes
of rational bundles), and compute their images in $A(\overline Q\times\overline Q)$ under the map of formula~\eqref{neza}.

To compute these images we use the $T$-equivariant theory $A_T$, the injection~\eqref{ine} and the concrete
formulas for pullbacks and pushforwards for $A_T$ described in Section~\ref{seceq}. In particular, we need to know how
the fixed points on $\overline Q$, $\OGr(1,2,\overline Q)$ and $\mathbb{P}(\tau_2)$ look like and we need a description
of the weights of the tangent bundles at the fixed points. Below we will give this explicit description.

This reduces the computation of projectors in $A(Q\times Q)$ to combinatorics, which can be performed on a computer.
\end{ntt}

Let $(V,q)$ be a $2l$-dimensional {\it split} quadratic space with basis
$$\{e_1,\ldots e_l,e_{-l},\ldots,e_{-1}\}$$ and $q(e_i)=0$,
$q(e_i+e_j)=\delta_{i,-j}$.
Let $Q$ be the corresponding projective split quadric and $G=\mathrm{O}^+(Q)$ be the respective orthogonal group
with a split maximal torus $T$ such that $\la e_i\ra$ is the weight subspace
of $V$ of weight $\chi_i$, i.e., $te_i=\chi_i(t)e_i$ for all $t\in T$,
$i=1,\ldots,l$. The character group of $T$ is $\zz\chi_1\oplus\ldots\oplus\zz\chi_l$.

The torus $T$ acts naturally on
$Q$, $\OGr(1,2,Q)$, and $\mathbb{P}(\tau_2)$ with finite number of fixed points.
The group $G$ acts on these varieties, so the Weyl group $W$ of $G$ acts on the
respective fixed points. If $x$ is a fixed point and the weights of $T_{X,x}$
are $\rho_1,\ldots,\rho_{\dim X}$ (with multiplicities), then the weights
of $T_{X,w(x)}$ are $w(\rho_1),\ldots,w(\rho_{\dim X})$, $w\in W$.
So, it is sufficient to compute the weights for representatives in $W$-orbits
of fixed points. For $X=G/P$, where $P$ is a parabolic subgroup of $G$, the weights at the point $1\cdot P$ are the roots
of the unipotent radical of $P^-$.

We have a commutative diagram
$$\xymatrix{
A_T^*(Q)\oplus A_T^{*-1}(\OGr(1,2,Q))\oplus A_T^{*-\dim Q}(Q)\ar[r]^-{\simeq}\ar@{^{(}->}[d]&A_T^*(Q\times Q)\ar@{^{(}->}[d]\\
\displaystyle\bigoplus_{x\in Q^T}A^*_T(x)\oplus\bigoplus_{y\in\OGr(1,2,Q)^T}A^{*-1}_T(y)
\oplus\bigoplus_{x'\in Q^T}A^{*-\dim Q}_T(x')\ar[r]&\displaystyle\bigoplus_{z\in(Q\times Q)^T}A^*_T(z)
}
$$

\begin{ntt}[Fixed points on a quadric]
The fixed points on $Q$ are the lines $\la w(e_1)\ra$, $w\in W$, i.e., the
lines $\la e_i\ra$, $i=-l,\ldots,-1,1,\ldots l$.

The weights at $\la e_1\ra$ are $\pm \chi_2-\chi_1,\ldots,\pm \chi_l-\chi_1$.
\end{ntt}

\begin{ntt}[Fixed points on $\OGr(1,2,Q)$]
The fixed points on $Q$ are the flags
$$\big(\la w(e_1)\ra\le\la w(e_1),w(e_2)\ra\big),\quad w\in W.$$

The weights at $\big(\la e_1\ra\le\la e_1,e_2\ra\big)$ are
$$\pm \chi_2-\chi_1,\pm \chi_3-\chi_1,\ldots,\pm \chi_l-\chi_1,\pm \chi_3-\chi_2,
\pm \chi_4-\chi_2,\ldots,\pm \chi_l-\chi_2.$$
\end{ntt}

\begin{ntt}[Fixed points on $\mathbb{P}(\tau_2)$]
There are two $W$-orbits of fixed points on $\mathbb{P}(\tau_2)$:
$$\big(\la e_1\ra\le\la e_1,e_2\ra,\la e_1\ra\le\la e_1,e_2\ra\big)\text{ and }
\big(\la e_1\ra\le\la e_1,e_2\ra,\la e_2\ra\le\la e_1,e_2\ra\big).$$

The weights at the first (resp. at the second) of these points are the same as for $\OGr(1,2,Q)$
at $\big(\la e_1\ra\le\la e_1,e_2\ra\big)$ together with the weight of
$\mathbb{P}^1$ at $\la e_1\ra$ (resp. at $\la e_2\ra$), which is $\chi_2-\chi_1$
(resp. $\chi_1-\chi_2$).
\end{ntt}

$$\xymatrix{
A_T^{*-1}(\OGr(1,2,Q))\ar[r]^-{\pi^*}\ar@{^{(}->}[d]&A_T^{*-1}(\mathbb{P}(\tau_2))
\ar[r]^-{f_*}\ar@{^{(}->}[d]&A_T^*(Q\times Q)\ar@{^{(}->}[d]\\
\displaystyle\bigoplus_{y\in\OGr(1,2,Q)^T}A^{*-1}_T(y)\ar[r]&\displaystyle
\bigoplus_{u\in\mathbb{P}(\tau_2)^T}A^{*-1}_T(u)\ar[r]&\displaystyle\bigoplus_{z\in (Q\times Q)^T}A^*_T(z)}
$$

The map $\pi\colon\mathbb{P}(\tau_2)\to\OGr(1,2,Q)$ is on the fixed points $2:1$ and the points
$$\big(\la e_1\ra\le\la e_1,e_2\ra,\la e_1\ra\le\la e_1,e_2\ra\big)\text{ and }
\big(\la e_1\ra\le\la e_1,e_2\ra,\la e_2\ra\le\la e_1,e_2\ra\big)$$ map to
$\la e_1\ra\le\la e_1,e_2\ra$.

Then the map $f$ maps $$\big(\la e_1\ra\le\la e_1,e_2\ra,\la e_1\ra\le\la e_1,e_2\ra\big)
\text{ to }(\la e_1\ra,\la e_1\ra)$$ and $$\big(\la e_1\ra\le\la e_1,e_2\ra,\la e_2\ra\le\la e_1,e_2\ra\big)
\text{ to }(\la e_1\ra,\la e_2\ra).$$ So, the point $(\la e_1\ra,\la e_1\ra)\in(Q\times Q)^T$ has $2l-2$ preimages
and $(\la e_1\ra,\la e_2\ra)$ has one preimage.

By formula~\eqref{fff} applied to $f$ and $z:=(\la e_1\ra,\la e_2\ra)$ we obtain
$$\tilde f_*(a)_z=a_u\cdot(-\chi_1-\chi_2),$$
where $u=\big(\la e_1\ra\le\la e_1,e_2\ra,\la e_2\ra\le\la e_1,e_2\ra\big)$.

Applying $w\in W$ we cover all coordinates of $\tilde f_*(a)$
in the $W$-orbit of $(\la e_1\ra,\la e_2\ra)$. Note that the fixed points in the
orbit of $(\la e_1\ra,\la e_{-1}\ra)$ have no preimages under $f$. So, these coordinates
of $\tilde f_*(a)$ equal $0$.

It remains to cover coordinates in the orbit of $z:=(\la e_1\ra,\la e_1\ra)\in (Q\times Q)^T$.
We have $f^{-1}(\la e_1\ra,\la e_1\ra)$ is isomorphic to the $(2l-2)$-dimensional subquadric $Z$ of $Q$
given by the restriction of $q$ to $\la e_2,\ldots,e_l,e_{-l},\ldots,e_{-2}\ra$.
The embedding $j\colon Z\to\mathbb{P}(\tau_2)$ is given by
$j(\la v\ra)=\big(\la e_1\ra\le\la e_1,v\ra,\la e_1\ra\le\la e_1,v\ra\big)$.
By formula~\eqref{fff}
$$\tilde f_*(a)_z=\zeta_*(j^*(a)\cdot(-\chi_1-\chi_2)\cdot\prod_{i=3}^l(\chi_i-\chi_1)(-\chi_i-\chi_1)).$$
This finishes the description of the fixed points, weights and push-forwards needed to apply our method~\ref{method}.

\bibliographystyle{chicago}

\medskip

\medskip

\noindent
{\sc Victor Petrov\\
St.~Petersburg Department of Steklov Mathematical Institute, Russian Academy of Sciences,
Fontanka 27, 191023 St.~Petersburg, Russia

\medskip

\noindent
Chebyshev Laboratory, St. Petersburg State University, 14th Line 29b, Vasilyevsky Island, Saint Petersburg, Russia
}

\medskip


\noindent
{\sc Nikita Semenov\\
Institut f\"ur Mathematik, Johannes Gutenberg-Universit\"at
Mainz, Staudingerweg 9, D-55128, Mainz, Germany}

\noindent
{\tt semenov@uni-mainz.de}

\end{document}